\documentclass{amsart}
\usepackage{amsfonts,amssymb,amscd,amsmath,enumerate,verbatim,calc}
\usepackage{xy}
\usepackage{mathtools}

\newcommand{\CM}{Cohen-Macaulay}

\newcommand{\rt}{\rightarrow}

\newcommand{\sub}{\subseteq}

\theoremstyle{plain}

\newtheorem{theorem}{Theorem}[section]
\newtheorem{corollary}[theorem]{Corollary}
\newtheorem{lemma}[theorem]{Lemma}

\theoremstyle{definition}
\newtheorem{definition}[theorem]{Definition}
\newtheorem{s}[theorem]{}
\newtheorem{remark}[theorem]{Remark}

\theoremstyle{remark}

\begin{document}
	\title[Associated graded modules]{On associated graded modules of maximal Cohen-Macaulay modules over hypersurface rings }
	\author{Ankit Mishra}
	\email{ankitmishra@math.iitb.ac.in}
	
	\author{ Tony~J.~Puthenpurakal}
	\email{tputhen@math.iitb.ac.in}
	
	\address{Department of Mathematics, IIT Bombay, Powai, Mumbai 400 076}

	\date{\today}
	
	
	\begin{abstract}
		Let $A=Q/(f)$ where $(Q,\mathfrak{n})$ be a complete regular local ring of dimension $d+1$, $f\in \mathfrak{n}^i\setminus\mathfrak{n}^{i+1}$ for some $i\geq 2$ and $M$ an MCM $A-$module with $e(M)=\mu(M)i(M)+1$ then we  prove that depth $G(M)\geq d-1$. If $(A,\mathfrak{m})$ is a complete hypersurface ring of dimension $d$ with infinite residue field and $e(A)=3$, let $M$ be an MCM $A$-module with $\mu(M)=2$ or $3$ then we  prove that depth $G(M)\geq d-\mu(M)+1$.  Our paper is the first systematic study of depth of associated graded modules of MCM modules over hypersurface rings
			\end{abstract}
	\maketitle
	\section{Introduction}
	
	Let $(A,\mathfrak{m})$ be Noetherian local ring of dimension $d$ and $M$ a finite \CM\ $A$-module of dimension $r$. Let $G(A)=\bigoplus_{n\geq0}\mathfrak{m}^n/\mathfrak{m}^{n+1}$ be associated graded ring of $A$ with respect to $\mathfrak{m}$ and $G(M)=\bigoplus_{n\geq0}{\mathfrak{m}^nM}/{\mathfrak{m}^{n+1}M}$ be associated graded module of $M$ with respect to $\mathfrak{m}$. Now $\mathcal{M}=\bigoplus_{n\geq1}\mathfrak{m}^n/\mathfrak{m}^{n+1}$ is irrelevant maximal ideal of $G(A)$ we set depth $G(M)$=grade$(\mathcal{M},G(M))$. If $L$ be an $A$-module then minimal number of generators of $L$ is denoted by  $\mu(L)$ and its length is denoted by $\ell(L)$.
	
	We know that Hilbert-Samuel function of $M$ with respect to $\mathfrak{m}$ is
	$$H^1(M,n)=\ell({M}/{\mathfrak{m}^{n+1}M})\ \text{for all}\ n\geq0.$$
	There exists a polynomial $P_M(z)$ of degree $r$  such that
	  $$H^1(M,n)=P_M(n)\ \text{for}\ n\gg0.$$
	This polynomial can be written as $$P_M(X)=\sum_{i=0}^{r}(-1)^ie_i(M)\binom{X+r-i}{r-i}$$
	These coefficients $e_i(M)'$s are integers and known as {\it Hilbert coefficients} of $M$.
	
	We know that Hilbert series of $M$ is formal power series $$H_M(z)=\sum_{n\geq0}\ell(\mathfrak{m}^nM/\mathfrak{m}^{n+1}M)z^n$$
	We can write $$ H_M(z)=\frac{h_M(z)}{(1-z)^r},\ \text{where}\ r=dimM $$
	
	Here, $h_M(z)=h_0(M)+h_1(M)z+\ldots+h_s(M)z^s\in \mathbb{Z}[z]$ and $h_M(1)\neq0$. This polynomial is know as {\it h-polynomial} of $M$.

	If we set $f^{(i)}$ to denote $i$th formal derivative of  a polynomial $f$ then it is easy to see that $e_i(M)=h_M^{(i)}(1)/i!$ for $i=0,\ldots,r$.  It is  convenient to set $e_i(M)=h_M^{(i)}(1)/i!$ for all $i\geq0.$
	
	We know that if $(A,\mathfrak{m})$ is \CM\ with red$(A)\leq2$ then $G(A)$ is \CM\ (see\cite[Theorem 2.1]{S}). 
	
	If $M$ is a \CM \ $A$-module with red$(M)\leq1$ then  $G(M)$ is \CM \ (see \cite[Theorem 16]{Pu0}), but if red$(M)=2$, then $G(M)$ need not be \CM\ (see \cite[Example 3.3]{PuMCM}).
	
	Here  we consider maximal \CM\ (MCM) modules over a \CM\ local ring $(A,\mathfrak{m})$. We know that if $A$ is a regular local ring then $M$ is free, say $M\cong A^s$. This implies $G(M)\cong G(A)^s$ is \CM.
	
	The next case is when $A$ is a hypersurface ring. {\it For convenience we assume $A=Q/(f)$ where $(Q,\mathfrak{n})$ is a regular local ring with infinite residue field and $f\in \mathfrak{n}^2$.} 
	
	If $f\in \mathfrak{n}^2\setminus\mathfrak{n}^3$ then $A$ has minimal multiplicity. It follows that any MCM module $M$ over $A$ has minimal multiplicity. So $G(M)$ is \CM.
	
	One of the cases of interest for us was when  $f\in \mathfrak{n}^3\setminus\mathfrak{n}^4$. Note in this case red$(A)=2$. So if $M$ is any MCM $A$-module then red$(M)\leq2$. In this case $G(M)$ need not \CM\ (see  \cite[Example 3.3]{PuMCM}).
	
	We know that if red$(M)\leq 1$ then $M$ has minimal multiplicity. This implies $G(M)$ is \CM\ (see \ref{minmulM}). Furthermore, if red$(M)=0$ then $M$ is an Ulrich module ({\it an MCM module $M$ is said to be Ulrich module if $e(M)=\mu(M)$}).
	
	Notice if $M$ is an MCM module over $A$ then projdim$_Q(M)=1$. So, $M$ has a minimal presentation over $Q$ 
	$$0\rt Q^{\mu(M)}\xrightarrow{\phi}Q^{\mu(M)}\rt M\rt 0.$$
	We investigate $G(M)$ in terms of invariants of a minimal presentation of $M$ over $Q$. \\
	Set $i(M)=$ max\{$i|$all entries of  $\phi$ are in $\mathfrak{n}^i$\}. Then
	from \cite[Theorem 2]{PuMCM} we know that $e(M)\geq \mu(M)i(M)$ for any MCM module $M$ over a hypersurface ring; in that paper, it is given that if $e(M)=\mu(M)i(M)$ then $G(M)$ is \CM.
	
	Here we  consider the case when   $e(M)=\mu(M)i(M)+1$ and prove that:
	\begin{theorem}\label{thm1.1}
		Let $(Q,\mathfrak{n})$ be a complete regular local ring of dimension $d+1$ with infinite residue field. Let $g\in \mathfrak{n}^i\setminus\mathfrak{n}^{i+1}$ with $i\geq 2$. Let $(A,\mathfrak{m})=(Q/(g),\mathfrak{n}/(g))$  and $M$ be an MCM $A$-module. Now  if $e(M)=\mu(M)i(M)+1$ then depth$G(M)\geq d-1$ and $h_M(z)=\mu(M)(1+z+\ldots+z^{i(M)-1})+z^s$ where $s\geq i(M)$. Furthermore, $G(M)$ is \CM\ if and only if $s=i(M)$.  
	\end{theorem}

	If $\mu(M)=r$ and $det(\phi) \in \mathfrak{n}^r\setminus\mathfrak{n}^{r+1}$  then we know that $e(M)=\mu(M)$ (see \cite[Theorem 2]{PuMCM}). So $M$ is an Ulrich module. This implies $G(M)$ is \CM. Here we consider the case when $det(\phi) \in \mathfrak{n}^{r+1}\setminus\mathfrak{n}^{r+2}$ and  prove 
	\begin{corollary}\label{4}
		Let $({Q},\mathfrak{n})$ be a complete regular local ring with infinite residue field of dimension $d+1$ with $d\geq 0$. Let $M$ be a $Q$-module with minimal presentation $$0\rt Q^r\xrightarrow{\phi} Q^r \rt M \rt 0$$
		Now if $\phi = [a_{ij}]
		$
		where $a_{ij} \in \mathfrak{n}$ with  $f=det(\phi) \in \mathfrak{n}^{r+1}\setminus \mathfrak{n}^{r+2}$, then depth$G(M)\geq d-1$. In this case if  $red(M)\leq2$  we can also prove that
		\begin{enumerate}
			\item $G(M)$ is \CM \  if and only if $h_M(z)=r+z$.
			\item depth$G(M)=d-1$ if and only if $h_M(z)=r+z^2$.
		\end{enumerate}
	\end{corollary}

	\begin{remark}
		Let $(A,\mathfrak{m})$ be complete hypersurface ring of dimension $d$ and $M$ be an MCM $A-$module with $\mu(M)=1$. Then we can write $A=Q/(f)$ where $(Q,\mathfrak{n})$ be a regular local ring of dimension $d+1$ and $f\in \mathfrak{n}^i\setminus\mathfrak{n}^{i+1}$. Since $\mu(M)=1$, $M$ has a minimal presentation $0\rt Q\xrightarrow{a} Q\rt M\rt 0$ where $a\in \mathfrak{n}$. This implies $M\cong Q/(a)Q$, so $G(M)$ is \CM.
	\end{remark}

	Now we consider the case when $\mu(M)=2$ and prove 
	\begin{theorem}\label{1}
		Let $(A,\mathfrak{m})$ be a complete hypersurface ring  of dimension $d$ with $e(A)=3$ and infinite residue field. Let  $M$ be  an MCM $A$-module with  $\mu(M)=2$, 
		 then depth$G(M)\geq d-1$.
	\end{theorem}
	
	The next theorem deals with the case when $\mu(M)=3 $ and we have proved that
	
	\begin{theorem}\label{2}
		Let $(A,\mathfrak{m})$ be a complete hypersurface ring  of dimension $d$ with $e(A)=3$ and infinite residue field. Let  $M$ be an MCM $A$-module with   $\mu(M)=3$, then depth$G(M)\geq d-2$.
	\end{theorem}

Here is an overview of the contents of this paper. In  section 2, we give some preliminary which we have used in the paper.  In section 3, we prove Theorem \ref{thm1.1}, and as its corollary we prove  Corollary \ref{4}. In section 4, we discuss $\mu(M)=2$ case and prove Theorem \ref{1}. In section 5, we discuss $\mu(M)=3$ case and prove Theorem \ref{2}.  In the last section examples are given.

\section{Priliminaries}
Let $(A,\mathfrak{m})$ be a Noetherian local ring of dimension $d$, and $M$ an $A$-module of dimension $r$.
\s An element $x\in \mathfrak{m}$ is said to be a superficial element of $M$ if there exists an integer $n_0>0$ such that $$(\mathfrak{m}^nM:_Mx)\cap \mathfrak{m}^{n_0}M=\mathfrak{m}^{n-1}M\ \text{for all}\ n>n_0$$
We know that if residue field $k=A/\mathfrak{m}$ is infinite then superficial elements always exist (see \cite[Pg 7]{Sbook}). A sequence of elements $x_1,\ldots,x_m$ is said to be superficial sequence if $x_1$ is $M$-superficial and $x_i$ is $M/(x_1,\ldots,x_{i-1})M$-superficial for $i=2,\ldots,m.$

\begin{remark}
	
	\begin{enumerate}
		\item 	If $x$ is $M-$superficial and regular then we have $(\mathfrak{m}^nM:_Mx)=\mathfrak{m}^{n-1}M$ for all $n\gg 0.$
		
		\item If depth$M>0$ then it is easy to show that  every $M$-superficial element is also $M-$ regular.
	\end{enumerate}

\end{remark}

\s \label{Base change} Let $f:(A,\mathfrak{m})\rt (B,\mathfrak{n})$ be a flat local  ring homomorphism with $\mathfrak{m}B=\mathfrak{n}$. If $M$ is an $A$-module set $M'=M\otimes_A B$, then  following facts are well known 
\begin{enumerate}
	\item $H(M,n)=H(M',n)$ for all $n\geq0$.
	\item depth$_{G(A)}G(M)=$depth$_{G(A')}G(M')$.
	\item projdim$_AM$=projdim$_{A'}M'$
	
\end{enumerate}
We will use this result in the following two cases:
\begin{enumerate}
	\item We can assume $A$ is complete by taking $B=\hat{A}$.
	\item We can assume the residue field of $A$ is infinite, because if the residue field $(k=A/\mathfrak{m})$ is finite we can take $B=A[X]_S$  where $S=A[x]\setminus \mathfrak{m}A[X]$. Clearly, the residue field of $B=k(X)$ is infinite.
\end{enumerate}

Since all the properties we deal in this article are invariant when we go from $A$ to $A'$, we can assume that residue field of $A$ is infinite.

\s If $a$ is a non-zero element of $M$ and if $i$ is the largest integer such that $a\in \mathfrak{m}^iM$, then we denote image of $a$ in $\mathfrak{m}^i \  M/\mathfrak{m}^{i+1} \ M$ by $a^*$. If $N$ is a submodule of $M$, then $N^*$ denotes the graded submodule of $G(M)$ generated by all $b^*$ with $b\in N$.

\begin{definition}
	Let $(A,\mathfrak{m})$ be a Noetherian local ring and $M\ne 0$ be a finite $A$-module then $M$ is said to be a \CM\ $A$-module if depth $M=$dim $M$, and a maximal \CM\ (MCM) module if depth $M=$dim $A$.
\end{definition}

\s \label{mod-sup} If $x\in \mathfrak{m}\setminus\mathfrak{m}^2$. Set $N=M/xM$ and $K=\mathfrak{m}/(x)$ then we have {\bf Singh's equality}\index{Singh's equality} ( for $M=A$ see \cite[Theorem 1]{singh}, and for the module case see \cite[Theorem 9]{Pu0})
$$H(M,n)=\ell(N/K^{n+1}N)-\ell\left(\frac{\mathfrak{m}^{n+1}M:x}{\mathfrak{m}^nM}\right)\ \text{for all}\ n\geq0.$$

Set $b_n(x,M)=\ell(\mathfrak{m}^{n+1}M:x/\mathfrak{m}^nM)$ and $b_{x,M}(z)=\sum_{n\geq0}b_n(x,M)z^n$. Notice that $b_0(x,M)=0$. 

\s \label{Property} (See \cite[Corollary 10]{Pu0}) Let $x\in \mathfrak{m}$ be an $M-$superficial and regular element. Set $B=A/(x)$, $N=M/xM$ and $K=\mathfrak{m}/(x)$ then we have
\begin{enumerate}
	\item dim$M-1$ = dim$N$ and $h_0(N)=h_0(M)$.
	\item $b_{x,M}$ is a polynomial.
	\item $h_M(z)=h_N(z)-(1-z)^rb_{x,M}(z)$.
	\item $h_1(M)=h_1(N)$ if and only if $\mathfrak{m}^2M\cap xM=x\mathfrak{m}M.$
	\item $e_i(M)=e_i(N)$ for $i=0,\ldots,r-1.$
	\item $e_r(M)=e_r(N)-(-1)^r\sum_{n\geq0}b_n(x,M).$
	\item $x^*$ is $G(M)$-regular if and only if $b_n(x,M)=0$ for all $n\geq0.$
	\item $e_r(M)=e_r(N)$ if and only if $x^*$ is $G(M)$-regular.
	\item depth $G(M)\geq1$ if and only if $h_{M}(z)=h_N(z)$.
\end{enumerate}
\s \label{Sally-des} {\bf Sally-descent }(see \cite[Theorem 8]{Pu0}): Let $(A,\mathfrak{m})$ be a \CM\ local ring of dimension $d$ and $M$ be \CM\ module of dimension $r$. Let $x_1,\ldots,x_c$ be a $M$-superficial sequence with $c\leq r$. Set $N=M/(x_1,\ldots,x_c)M$ then   

depth $G(M)\geq c+1$ if and only if depth $G(N)\geq 1$. 

\s Let $(A,\mathfrak{m})$ be Noetherian local ring, $M$ a finitely generated  $A$-module. Let $J \sub \mathfrak{m}$ be an ideal, then $J$ is said to be a reduction  of $M$ if $\mathfrak{m}^{n+1}M=J\mathfrak{m}^nM$ for some $n\geq 0$. Set red$_J(M)=$ min\{$n | \mathfrak{m}^{n+1}M=J\mathfrak{m}^nM$\}.
\begin{definition}
	A reduction is called minimal reduction if it is minimal with respect to the inclusion.
\end{definition}

\s Assume the residue field of $A$ is infinite. The reduction number of $M$ (red$(M)$) is defined as red$(M)=$ min\{ red$_J(M)| J$ is a minimal reduction of $M$\}.

\begin{definition}\label{Ulrich}
	Let $(A,\mathfrak{m})$ be a Noetherian local ring and $M$ be a maximal \CM \ module  then $M$ is said to be a Ulrich module if $e(M)=\mu(M)$.

\end{definition}

\begin{definition}(\cite[Definition 15]{Pu0})
	Let $(A,\mathfrak{m})$ be a Noetherian local ring and $M$ a \CM\ $A$-module. We say $M$ has minimal multiplicity if $e(M)=h_0(M)+h_1(M)$.
\end{definition}
\begin{remark}\label{minmulM}(\cite[Theorem 16]{Pu0})
	Let $(A,\mathfrak{m})$ be a Noetherian local ring and $M$ a \CM\ $A$-module. If $M$ has minimal multiplicity then $G(M)$ is \CM\ $G(A)$-module. Also, $h_M(z)=h_0(M)+h_1(M)z$.\\
	If $A/\mathfrak{m}$ is infinite then $M$ has minimal multiplicity if and only if red$(M)\leq 1$ (see \cite[Theorem 16]{Pu0}).
\end{remark}

\s\label{dm1hpol}(see \cite[Proposition 13]{Pu0}) Let $(A,\mathfrak{m})$ be a Noetherian local ring and $M$ a \CM\ $A$-module with dim$M=1$. Let $x$ be an $M$-superficial element. Set $\rho_n(M)=\ell(\mathfrak{m}^{n+1}M/x\mathfrak{m}^nM)$ for all $n\geq 0$. If deg$h_M(z)=s$ then \\ $\rho_n(M)=0$ for all $n\geq s$, and $$h_M(z)=h_0(M)+\sum_{i=0}^{s}(\rho_{i-1}(M)-\rho_i(M))z^i.$$

\s\label{e2>} (see \cite[Proposition 3.1]{Rossi}) Let $(A,\mathfrak{m})$ be a Noetherian local ring and $M$ a \CM\ $A$-module then $e_2(M)\geq 0$.

\s  (See \cite[section 6]{heinzer}) For any $n\geq1$ we can define Ratliff-Rush submodule of $M$ associated with $\mathfrak{m}^n$ as 
$$\widetilde{\mathfrak{m}^nM}=\bigcup_{i\geq0}(\mathfrak{m}^{n+i}M:_M\mathfrak{m}^i)$$ The filtration $\{\widetilde{\mathfrak{m}^nM}\}_{n\geq1}$ is known as the Ratliff-Rush filtration \index{Ratliff-Rush filtration} of $M$ with respect to $\mathfrak{m}$.

For the proof of the following properties in the ring case see \cite{Ratliff}. This proof can be easily extended for the modules. Also see \cite[2.2]{Naghipour}.

\s If depth$M>0$ and  $x\in \mathfrak{m}$ is a $M-$superficial element then we have
\begin{enumerate}
	\item $\widetilde{\mathfrak{m}^nM}=\mathfrak{m}^nM$ for all $n\gg0.$
	\item  $(\widetilde{\mathfrak{m}^{n+1}M}:x)=\widetilde{\mathfrak{m}^nM}$ for all $n\geq1.$
\end{enumerate}
\s\label{htilde} If depth$M >0$. Let $\widetilde{G(M)}=\bigoplus_{n\geq0}\widetilde{\mathfrak{m}^nM}/\widetilde{\mathfrak{m}^{n+1}M}$ be the associated graded module  of $M$ with respect to Ratliff-Rush filtration. Then its Hilbert series
$$\sum_{n\geq0}\ell(\widetilde{\mathfrak{m}^nM}/\widetilde{\mathfrak{m}^{n+1}M})z^n=\frac{\widetilde{h_M}(z)}{(1-z)^r}$$
Where $\widetilde{h_M}(z)\in \mathbb{Z}[z]$. Set $r_M(z)=\sum_{n\geq0}\ell(\widetilde{\mathfrak{m}^{n+1}M}/\mathfrak{m}^{n+1}M)z^n$; clearly, $r_M(z)$ is a polynomial with non-negative integer coefficients (because depth$M>0$). Now we have $$h_M(z)=\widetilde{h_M}(z)+(1-z)^{r+1}r_M(z);\ \text{where }\ r=\text{dim}M$$
We know that depth$G(M)>0$ if and only if $r_M(z)=0$.

\begin{definition}\label{hypersurface}
	Let $(A,\mathfrak{m})$ be a Noetherian local ring, then $A$ is said to be a hypersurface ring if its completion can be written as a quotient of a regular local ring by a principal ideal.
\end{definition}
\s Let $(Q,\mathfrak{n})$ be a regular local ring, $f\in \mathfrak{n}^e\setminus\mathfrak{n}^{e+1}$ and $A=Q/(f)$. If $M $ is an MCM $A-$module then projdim$_Q(M)=1$ and $M$ has a minimal presentation: $$0\rt Q^{\mu(M)}\rt Q^{\mu(M)}\rt M \rt 0$$

\s \label{i(M)}Let $(Q,\mathfrak{n})$ be a regular local ring and $\phi : Q^t\rt Q^t$  a linear map,  set
$$i_\phi=\text{max}\{i |\  \text{all entries of}\ \phi \ \text{are in }\ \mathfrak{n}^i\}$$
If $M $ has minimal presentations: $0\rt Q^t\xrightarrow{\phi}Q^t\rt M\rt0$ and
$0\rt Q^t\xrightarrow{\phi'}Q^t\rt M\rt0$ then it is well known that $i_\phi=i_{\phi'}$ and det$(\phi)=u$det$(\phi')$ where $u$ is a unit. We set $i(M)=i_\phi$ and det$M=$det$(\phi)$. For any non-zero element $a$ of $Q$ we set $v_Q(a)=max\{i|a\in \mathfrak{n}^i\}$. We are choosing this set-up from \cite{PuMCM}.

\begin{definition} \label{phi}
	(See \cite[Definition 4.4]{PuMCM}) Let $(Q,\mathfrak{n})$ be a regular local ring, $A=Q/(f)$ where $f\in \mathfrak{n}^e\setminus\mathfrak{n}^{e+1}, e\geq2$ and $M$  an MCM $A-$module with minimal presentation: $$0\rt Q^t\xrightarrow{\phi}Q^t\rt M\rt0$$
	Then an element $x$ of $\mathfrak{n}$ is said to be $\phi-$ superficial \index{$\phi-$ superficial}if we have
	\begin{enumerate}
		\item $x$ is $Q\oplus A\oplus M$ superficial.
		\item If $\phi=(\phi_{ij})$ then $v_Q(\phi_{ij})=v_{Q/xQ}(\overline{\phi_{ij}})$.
		\item $v_Q(det(\phi))=v_{Q/xQ}det(\overline{\phi})$
	\end{enumerate}
\end{definition}
\begin{remark}

	If $x$ is $Q\oplus A\oplus M\oplus (\oplus_{ij}Q/(\phi_{ij}))\oplus Q/(det(\phi))-$superficial then it is $\phi-$superficial. So if the residue field of $Q$ is infinite then $\phi-$superficial elements always exist.
\end{remark}
\begin{definition}
	(See \cite[Definition 4.5]{PuMCM}) Let $(Q,\mathfrak{n})$ be a regular local ring, $A=Q/(f)$ where $f\in \mathfrak{n}^e\setminus\mathfrak{n}^{e+1}, e\geq2$ and $M$  an MCM $A-$module with minimal presentation: $$0\rt Q^t\xrightarrow{\phi}Q^t\rt M\rt0.$$
	We say that $x_1,\ldots,x_c$ is a $\phi$-superficial sequence if $\overline{x_n}$ is $(\phi \otimes_Q Q/(x_1,\ldots,x_{n-1}))$-superficial for $n=1,\ldots,c$.
\end{definition}

\s \label{d=1} With above set-up we have   
\begin{enumerate}
	\item (see \cite[Lemma 4.7]{PuMCM})  If dim$M=1$  then
	$$h_M(z)=\mu(M)(1+z+\ldots+z^{i(M)-1})+\sum_{i\geq i(M)}h_i(M)z^i$$
	$$ \text{with}\ h_i(M)\geq0 \ \forall \ i.$$
	
	\item (see \cite[Theorem 2]{PuMCM}) $e(M)\geq \mu(M)i(M)$ and if $e(M)=\mu(M)i(M)$ then 
	
	$G(M)$ is Cohen-Macaulay and $h_M(z)=\mu(M)(1+z+\ldots+z^{i(M)-1})$.
	
\end{enumerate}

\s\label{e(A)=3} Let $(A,\mathfrak{m})$ be a complete hypersurface ring with infinite residue field and dimension $d$. If  $e(A)=3$,   we can write $(A,\mathfrak{m})=(Q/(g),\mathfrak{n}/(g))$ where $(Q,\mathfrak{n})$ is a regular local ring of dimension $d+1$ and $g\in \mathfrak{n}^3\setminus\mathfrak{n}^4$. This implies that $h_A(z)=1+z+z^2$. Now for any maximal $A$-superficial sequence $\underline{x}=x_1,\ldots,x_d$ we have $\mathfrak{m}^3=(\underline{x})\mathfrak{m}^2$. Let $M$ be an MCM $A$-module and $\underline{y}=y_1,\ldots,y_d$ be any maximal $A\oplus M$-superficial sequence. Then    $\mathfrak{m}^3M=(\underline{y})\mathfrak{m}^2M$. In particular, for any maximal $\phi-$superficial sequence $\underline{y}$ we have red$_{(\underline{y})}(M)\leq 2$. So red$(M)\leq 2$.

\s \label{exact seq}Let $(A,\mathfrak{m})$ be a \CM\ local ring  and $M$ a \CM\ $A$-module of dimension 2. Let $x,y$ be a maximal $M$-superficial sequence. 

Set $J=(x,y)$ and $\overline{M}=M/xM$ then we have exact sequence (for $M=A$ see \cite[Lemma 2.2]{rv})
\begin{align*}
0 \rt \mathfrak{m}^nM:J/\mathfrak{m}^{n-1}M\xrightarrow{f_1} \mathfrak{m}^nM:x/\mathfrak{m}^{n-1}M & \xrightarrow{f_2} \mathfrak{m}^{n+1}M:x/\mathfrak{m}^{n}M\\
\xrightarrow{f_3} \mathfrak{m}^{n+1}M/J\mathfrak{m}^nM
&\xrightarrow{f_4} \mathfrak{m}^{n+1}\overline{M}/y\mathfrak{m}^n\overline{M}\rt 0
\end{align*}
Here, $f_1$ is inclusion map, $f_2(a+\mathfrak{m}^{n-1}M)=ay+\mathfrak{m}^nM, f_3(b+\mathfrak{m}^nM)=bx+J\mathfrak{m}^nM$ and $f_4$ is reduction modulo $x$.

\s \label{exact d}Let $(A,\mathfrak{m})$ be a \CM\ local ring of dimension $d\geq 1$ and $M$ a maximal \CM\ $A$-module. Let $\underline{x}=x_1,\ldots,x_d$ be a maximal $M$-superficial sequence. Set $N=M/x_1M$, $J=(x_1,\ldots,x_d)$ and $\overline{J}$ is image of $J$ is $A/(x_1)$. Then we have  
$$0 \rt \mathfrak{m}^{2}M:x_1/\mathfrak{m}M
\xrightarrow{f} \mathfrak{m}^{2}M/J\mathfrak{m}M
\xrightarrow{g} \mathfrak{m}^{2}{N}/\overline{J}\mathfrak{m}{N}\rt 0.$$
Here, $f(a+\mathfrak{m}M)=ax_1+J\mathfrak{m}M$ and $g$ is reduction modulo $x_1$.

\s\label{exact1} Let $(A,\mathfrak{m})$ be a \CM\ local ring of dimension one and $M$ a maximal \CM\ $A$-module. Let $x$ be a superficial element of $M$. Set $N=M/xM$. Then we have $$0\rt \mathfrak{m}^2M:x/\mathfrak{m}^2M\xrightarrow{f} \mathfrak{m}^2M/x\mathfrak{m}^2M\xrightarrow{g}\mathfrak{m}^2N/0\rt 0.$$
Here, $f(a+\mathfrak{m}^2M)=ax+x\mathfrak{m}^2M$ and $g$ is reduction modulo $x$.

The following  result is well known, but we will use this many times. For the convenience of the reader we state it

\s \label{overline{G(M)}} Let $(Q,\mathfrak{n},k)$ be a regular local ring of dimension $d+1$ and\\ $(A,\mathfrak{m})=(Q/(f),\mathfrak{n}/(f))$ with $f\in \mathfrak{n}^3\setminus\mathfrak{n}^{4}$. Now if $M$ is a maximal \CM\ $A$-module. Let $\underline{x}=x_1,\ldots,x_d$ be sufficiently general linear forms in $\mathfrak{n}/\mathfrak{n}^2$. Then $\underline{x}$ is $A\oplus Q\oplus M-$superficial sequence and red$_{(\underline{x})}(A)= 2$. So red$_{(\underline{x})}(M)\leq 2$. Set $S=G_{\mathfrak{n}}(Q)$, $R=S/(\underline{x^*})S$ then $R\cong k[T]$ and $G(A)/(\underline{x})G(A)\cong R/(T^s)$ for some $s\geq2$.\\ Now consider $\overline{G(M)}=G(M)/(\underline{x^*})G(M)$. Then
$$\overline{G(M)}=M/\mathfrak{m}M \oplus \mathfrak{m}M/(\mathfrak{m}^2M+(\underline{x})M) \oplus \mathfrak{m}^2M/(\mathfrak{m}^3M+(\underline{x})\mathfrak{m}M)$$
Its Hilbert series is $\mu(M)+\alpha z+\beta z^2$ where $\beta\leq \alpha\leq \mu(M)$, because it is an $R$-module which is also $R/(T^s)$-module and it is generated in degree zero.

\s \label{RR-2} Let $(A,\mathfrak{m})$ be a \CM\ local ring of dimension $d$ and $M$ be a finite $A$-module with depth$M\geq 2$. Let $x$ be an $M$-superficial element. Set $N=M/xM$. Then for $n\geq 0$ we have exact sequence  (see \cite[2.2]{Pu2}) 
$$0\rt \frac{(\mathfrak{m}^{n+1}M:x)}{\mathfrak{m}^nM}\rt \frac{\widetilde{\mathfrak{m}^nM}}{\mathfrak{m}^nM}\rt \frac{\widetilde{\mathfrak{m}^{n+1}M}}{\mathfrak{m}^{n+1}M} \rt \frac{\widetilde{\mathfrak{m}^{n+1}N}}{\mathfrak{m}^{n+1}N}. $$

In particular, we have exact sequence
$$0\rt \widetilde{\mathfrak{m}M}/\mathfrak{m}M\rt \widetilde{\mathfrak{m}N}/\mathfrak{m}N.$$

If depth$M=1$, then for all $n\geq 0$ we have following exact sequence 
$$0\rt \frac{(\mathfrak{m}^{n+1}M:x)}{\mathfrak{m}^nM}\rt \frac{\widetilde{\mathfrak{m}^nM}}{\mathfrak{m}^nM}\rt \frac{\widetilde{\mathfrak{m}^{n+1}M}}{\mathfrak{m}^{n+1}M}$$

{\bf Convention}:
Let $M$ be a maximal Cohen-Macaulay  module of dimension $d$ and $\underline{x}=x_1,\ldots,x_d$ be a maximal $\phi$-superficial sequence, then

$M_0=M$ and $M_t=M/(x_1,\ldots,x_t)M$ for $t=1,\ldots,d$. 

The following result is well-known.

\begin{lemma}\label{Md no free}
	Let $(A,\mathfrak{m})$ be a complete hypersurface ring of dimension $d$ with infinite residue field and multiplicity $e(A)=e$. Let  $M$ be a MCM module. Let $\underline{x}=x_1,\ldots,x_d$ be a maximal $A$-superficial sequence. If $M$ has no free summand, then $M_d=M/(x_1,\ldots,x_d)M$ also has no free summand.
\end{lemma}
\begin{proof}
	Since $M$ has no free summand there exists an MCM module $L$ such that $M=Syz_1^{A}(L)$ ( for instance see \cite[Theorem 6.1]{Eisenbud}). So we have  $0\rt M\rt F \rt L\rt 0$ where $F=A^{\mu(L)}$ and $M\sub \mathfrak{m}F$. Set $F_d=F/(\underline{x})F$. Going modulo $\underline{x}$ we get $M_d\sub \mathfrak{m}F_d$. Note $\underline{x}$ is also an $L-$regular sequence. So we have $\mathfrak{m}^{e-1}M_d\sub \mathfrak{m}^eF_d=0$, because red$(A)=e-1$. This implies $M_d$ has no free summand.
\end{proof}

\section{\bf The case when $e(M)=\mu(M)i(M)+1$}
 From \cite[theorem 2]{PuMCM} we know that for an MCM module over a hypersurface ring $e(M)\geq \mu(M)i(M)$ and if $e(M)= \mu(M)i(M)$ then $G(M)$ is \CM. Here we  consider the next case and prove that:
\begin{theorem}\label{em=mum}
	Let $(Q,\mathfrak{n})$ be a complete regular local ring of dimension $d+1$ with infinite residue field. Let $g\in \mathfrak{n}^i\setminus\mathfrak{n}^{i+1}$ with $i\geq 2$. Let $(A,\mathfrak{m})=(Q/(g),\mathfrak{n}/(g))$  and $M$ be a MCM $A$-module. Now  if $e(M)=\mu(M)i(M)+1$ then depth$G(M)\geq d-1$ and $h_M(z)=\mu(M)(1+z+\ldots+z^{i(M)-1})+z^s$ where $s\geq i(M)$. Furthermore, $G(M)$ is \CM\ if and only if $s=i(M)$.  
\end{theorem}
\begin{proof}
	If dim$M=0$, then $M\cong Q/(y^{i(M)})\oplus\ldots\oplus Q/(y^{i(M)})\oplus Q/(y^{i(M)+1}) $, because $e(M)=\mu(M)i(M)+1$ (see \cite[Remark 4.2]{PuMCM}). This implies $h_{M}(z)=\mu(M)(1+z+\ldots+z^{i(M)-1})+z^{i(M)}$.\\
	If dim$M=1$,  then from \ref{d=1}(1) we have
	$h_{M}(z)=\mu(M)(1+z+\ldots+z^{i(M)-1})+z^{s}$ for $s\geq i(M)$, because $e(M)=i(M)\mu(M)+1$. Let $x_1$ be a $\phi$-superficial element. Set $M_1=M/x_1M$. We know that $e(M)=e(M_1)$, $\mu(M)=\mu(M_1)$ and $i(M)=i(M_1)$. So from the dimension zero case $h_{M_1}(z)=\mu(M)(1+z+\ldots+z^{i(M)-1})+z^{i(M)}$. Therefore $G(M)$ is \CM \ if and only if $h_{M_1}(z)=h_M(z)$ (see \ref{Property}) if and only if $s=i(M)$. \\ 
	%We first consider the case when dim$M=2$ because if dim$M\leq1$ there is nothing to prove.\\
	If dim$M=2$.
	Let   $\underline{x}=x_1,x_2$ be a maximal $\phi$-superficial sequence (see \ref{phi}). Set $M_1=M/x_1M$, $M_2=M/\underline{x}M$, $J=(x_1,x_2)$ and $(Q',(y))=(Q/(\underline{x}),\mathfrak{n}/(\underline{x}))$.
	Clearly, $Q'$ is a DVR. Notice that since $\underline{x}$ is $\phi$-superficial sequence, $e(M)=e(M_1)=e(M_2),\mu(M)=\mu(M_1)=\mu(M_2)$ and $i(M)=i(M_1)=i(M_2)$. So $M_2\cong  Q'/(y^{i(M)})\oplus\ldots\oplus Q'/(y^{i(M)})\oplus Q'/(y^{i(M)+1}) $, because $e(M)=\mu(M)i(M)+1$. This implies $h_{M_2}(z)=\mu(M)(1+z+\ldots+z^{i(M)-1})+z^{i(M)}$. 
	
	Since dim$M_1=1$, $e(M)=e(M_1)$, $\mu(M)=\mu(M_1)$ and $i(M)=i(M_1)$, we get $e(M_1)=\mu(M_1)i(M_1)+1$. So from dimension one case,  $h-$polynomial of $M_1$ is
	$h_{M_1}(z)=\mu(M)(1+z+\ldots+z^{i(M)-1})+z^{s}$ for $s\geq i(M)$. \\
	Since $M=coker(\phi)$, for $n\leq i(M)-1$ we get $$\mathfrak{m}^nM/{\mathfrak{m}^{n+1}M}\cong \mathfrak{m}^n(Q)^{\mu(M)}/{\mathfrak{m}^{n+1}(Q)^{\mu(M)}}. $$ 
	Now if $h$-polynomial of $M$ is  $h_M(z)=h_0(M)+h_1(M)z+\ldots+h_t(M)z^t$ then 
	$$\ell(\mathfrak{m}^nM/\mathfrak{m}^{n+1}M)=\binom{n+2}{n}\mu(M)\ \text{for all}\  n\leq i(M)-1.$$
	So for all $ n\leq i(M)-1$ we have 
	\[
	(n+1)h_0(M)+nh_1(M)+\ldots +h_n(M)=\binom{n+2}{n}\mu(M). \tag{$\dagger$}
	\]
	Now since $h_0(M)=\mu(M)$, we get from ($\dagger$) $$h_0(M)=h_1(M)=\ldots=h_{i(M)-1}(M)=\mu(M)$$
	So we have, $h_n(M)=h_n({M_1})$ for all $n\leq i(M)-1$.\\
	From Singh's equality (\ref{mod-sup}) we have $$\mathfrak{m}^{n+1}M:x_1=\mathfrak{m}^nM\ \text{for}\ n=0,\ldots,i(M)-1.$$ So we have
	\begin{equation}\label{d211}
	\mathfrak{m}^{n+1}M\cap x_1M=x_1\mathfrak{m}^nM\ \text{for}\ n=0,\ldots,i(M)-1
	\end{equation}
	Since $h_n(M_1)= h_n(M_2)$ for $n=0,\ldots,i(M)-1$,
	from Singh's equality (\ref{mod-sup})
	\begin{equation}\label{d111}
	\mathfrak{m}^{n+1}{M_1}:x_2=\mathfrak{m}^n{M_1} \ \text{for}\ n=0,\ldots,i(M)-1
	\end{equation}
	Now we have  $\mathfrak{m}^{n+1}M\cap JM=J\mathfrak{m}^nM\ \text{for}\ n=0,\ldots,i(M)-1$.
	In fact, if $\alpha=ax_1+bx_2\in \mathfrak{m}^{n+1}M$. Going modulo $x_1$ we get $\overline{\alpha}=\overline{b}x_2\in \mathfrak{m}^{n+1}{M_1}$. From (\ref{d111}) we have $\overline{b}\in \mathfrak{m}^n{M_1}$. So we can write $b=c+f$, where $c\in \mathfrak{m}^nM$ and $f\in x_1M$. This implies  $\alpha=ax_1+cx_2+fx_2$. Hence $\alpha-cx_2\in \mathfrak{m}^{n+1}M\cap x_1M$. So from (\ref{d211}) we have $\alpha=cx_2+gx_1$ with $c,g\in \mathfrak{m}^nM$. This implies $\alpha\in J\mathfrak{m}^nM$.
	
	So we have
	\begin{equation}\label{VV11}
	vv_i(M)=\ell\Big(\frac{\mathfrak{m}^{i+1}M\cap JM}{J\mathfrak{m}^iM}\Big)=0 \ \text{for}\ i=0,\ldots,i(M)-1
	\end{equation}
	Since $\mathfrak{m}^{i(M)}M:x_1=\mathfrak{m}^{i(M)-1}M$,  we have from \ref{exact seq}
	\begin{equation}\label{v11}
	v_{i(M)-1}=\ell(\mathfrak{m}^{i(M)}M/J\mathfrak{m}^{i(M)-1}M)=\ell(\mathfrak{m}^{i(M)}{M_1}/x_2\mathfrak{m}^{i(M)-1}{M_1})=\rho_{i(M_1)-1}= 1
	\end{equation}
	Notice that last  equality in  \ref{v11} is clear from the $h$-polynomial of $M_1$. In fact, from the expression of $h_{M_1}$ and \ref{dm1hpol} we get $\rho_{i(M_1)-1}(M_1)=\rho_{i(M_1)}(M_1)=\ldots=\rho_{s-1}(M_1)=1$. Here $\rho_n(M_1)=\ell(\mathfrak{m}^{n+1}M_1/x_2\mathfrak{m}^nM_1)$ for all $n\geq 0$. \\
	Now from conditions (\ref{VV11}) and (\ref{v11}), depth$G(M)\geq 1$ (see \cite[Theorem 4.4]{Rossi}, take $p=i(M)-1$). So $h_M(z)=h_{M_1}(z)$ (see \ref{Property}). Also, $G(M)$ is \CM \ if and only if $s=i(M)$. \\
	Now assume dim$M\geq3$ and $\underline{x}=x_1,\ldots,x_d$ a maximal $\phi$-superficial sequence. Set $M_{d-2}=M/(x_1,\ldots,x_{d-2})M$. So, depth$G(M_{d-2})\geq 1$.\\
	By Sally-descent  we get depth$G(M)\geq d-1$ and $h_M(z)=\mu(M)(1+z+\ldots+z^{i(M)-1})+z^s$ where $s\geq i(M)$. Note that $G(M)$ is \CM\ if and only if $s=i(M)$.
\end{proof}

We know (from \cite[Theorem 2]{PuMCM}) that if $\mu(M)=r$ and $det(\phi) \in \mathfrak{n}^r\setminus\mathfrak{n}^{r+1}$, then $M$ is an Ulrich module. This implies $G(M)$ is \CM. Here we consider the case when $det(\phi)\in \mathfrak{n}^{r+1}\setminus \mathfrak{n}^{r+2}$. For $a\ne 0$, set  $v_Q(a)=max\{i|a\in \mathfrak{n}^i\}$.

\begin{corollary}\label{mu(r)}
	Let $({Q},\mathfrak{n})$ be a complete regular local ring with infinite residue field of dimension $d+1$ with $d\geq 0$. Let $M$ be a $Q$-module with minimal presentation $$0\rt Q^r\xrightarrow{\phi} Q^r \rt M \rt 0$$
	
	Now if $\phi = [a_{ij}]
	$
	where $a_{ij} \in \mathfrak{n}$ with  $f=det(\phi) \in \mathfrak{n}^{r+1}\setminus \mathfrak{n}^{r+2}$, then depth$G(M)\geq d-1$. In this case if    $red(M)\leq2$   we can also prove that
	\begin{enumerate}
		\item $G(M)$ is \CM \  if and only if $h_M(z)=r+z$.
		\item depth$G(M)=d-1$ if and only if $h_M(z)=r+z^2$
	\end{enumerate}
\end{corollary}
\begin{proof}
	Set $(A,\mathfrak{m})=(Q/(f),\mathfrak{n}/(f))$. Since $f.M=0$, this implies $M$ is an $A$-module. Also, it is clear that $M$ is an MCM $A$-module because projdim$_QM=1$.\\
	If dim$M=0$. Since $v_Q(det(\phi))=r+1 $ and $\phi$ is an $r\times r$-matrix. So we get $M\cong  Q/(y)\oplus\ldots\oplus Q/(y)\oplus Q/(y^2) $.
	This implies $h_{M}(z)=r+z$, where $r=\mu(M).$\\
	If dim$M\geq 1$. Let $x_1,\ldots,x_d$ be a maximal $\phi$-superficial sequence. Set $M_{d-1}=M/(x_1,\ldots,x_{d-1})M$,  $M_d=M/(x_1,\ldots,x_d)M$ and
	$(Q',(y))=(Q/(\underline{x}),\mathfrak{n}/(\underline{x}))$.\\
	Clearly, $Q'$ is a DVR.\\
	We know that $v_{Q'}(det(\phi\otimes Q')) = v_Q(det(\phi))=r+1 $ and $\phi$ is an $r\times r$-matrix. So $M_d\cong  Q'/(y)\oplus\ldots\oplus Q'/(y)\oplus Q'/(y^2) $.
	This implies $h_{M_d}(z)=r+z$, where $r=\mu(M).$\\
	In this case $i(M)=i(M_d)=1$, $e(M)=e(M_d)=r+1$ and $\mu(M)=\mu(M_d)=r$. 
	So we have, $e(M)=\mu(M)i(M)+1$. Now from the Theorem \ref{em=mum}, we get depth$G(M)\geq d-1$. This implies $h_{M_{d-1}}(z)=h_M(z)=r+z^s$ with $s\geq 1$ (for the first equality see \ref{Property}). Since red$M\leq 2$, we get $s$ = deg$h_M(z)\leq 2$ (see \ref{dm1hpol}). So here we have two cases.\\ First case when $s=1$. In this case $h_M(z)=r+z$ and $G(M)$ is \CM.\\
	Second case when $s=2$. In this case $h_M(z)=r+z^2$ and depth$G(M)=d-1$. 
\end{proof}

\section{\bf The case when $\mu(M)=2$ }
In this section we prove Theorem \ref{1}. We first consider the case when $M$ has no free summand and prove:
\begin{theorem}\label{muM=2}
	Let $(A,\mathfrak{m})$ be a complete hypersurface ring  of dimension $d$ with $e(A)=3$ and infinite residue field. Let  $M$ be  an MCM $A$-module with no free summand. Now if $\mu(M)=2$, 
	then depth$G(M)\geq d-1$.
	
\end{theorem}
\begin{proof} From \ref{e(A)=3} we can take
	 $(A,\mathfrak{m})=(Q/(f),\mathfrak{n}/(f))$, where
	$(Q,\mathfrak{n})$ is a regular local ring of dimension $d+1$ and $f\in \mathfrak{n}^3\setminus\mathfrak{n}^{4}$.\\
	Let dim$M\geq1$ and $0\rt Q^2\xrightarrow{\phi} Q^2\rt M\rt 0$ be a minimal presentation of $M$. Let $\underline{x}=x_1,\ldots,x_d$ be a maximal $\phi$-superficial sequence (see \ref{phi}). Set $M_d=M/\underline{x}M$ and $(Q',(y))=(Q/(\underline{x}),\mathfrak{n}/(\underline{x}))$.\\
	Clearly, $Q'$ is DVR and so $M_d\cong Q'/(y^{a_1})\oplus Q'/(y^{a_2})$. As $red_{(\underline{x})}(M)\leq2$ and $M_d$ has no free summand (see Lemma \ref{Md no free}), we can assume $1\leq a_1\leq a_2\leq 2$. We consider all possibilities separately 
	
	{\bf Case (1):} $a_1=a_2=1 $. \\
	In this case  $M_d\cong Q'/(y)\oplus Q'/(y)$. This implies $h_{M_d}(z)=2$ and 
	$e(M_d)=\mu(M_d)=2$. For dim$M\geq 1$ we know that $e(M)=e(M_d)$ and $\mu(M)=\mu(M_d)$. So $e(M)=\mu(M)=2$ and this implies  $M$ is Ulrich module. Therefore $G(M)$ is \CM \ and $h_M(z)=2$ (see \cite[Theorem 2]{PuMCM}).
	
	{\bf Case (2):} $a_1=1,a_2=2 $. \\
	In this case $M_d\cong Q'/(y)\oplus Q'/(y^2)$, so $h_{M_d}(z)=2+z$.\\
	This implies  $v_Q(det(\phi)) =v_{Q'}(det(\phi\otimes Q'))=3$. 
	Now for dim$M\geq 1$, from the Theorem \ref{mu(r)} depth$G(M)\geq d-1$. Also, we have two cases.\\ First case when  $h_M(z)=2+z$. In this case $G(M)$ is \CM.\\
	Second case when $h_M(z)=2+z^2$. In this case depth$G(M)=d-1$.

	{\bf Case(3):} $a_1=2,a_2=2$.\\
	In this case  $M_d\cong Q'/(y^2)\oplus Q'/(y^2)$. So $h_{M_d}(z)=2+2z$. This implies $e(M_d)=4=i(M_d)\mu(M_d)$.   Notice this equality is preserved modulo any $\phi$-superficial sequence, So for dim$M\geq1$,  $G(M)$ is \CM \ and $h_M(z)=2+2z$ (see \cite[Theorem 2]{PuMCM}). 
	\end{proof}

\begin{theorem}\label{mu2frsum}
	Let $(A,\mathfrak{m})$ be a complete hypersurface ring  of dimension $d$ with $e(A)=3$ and infinite residue field. Let  $M$ be  an MCM $A$-module with  free summand. Now if $\mu(M)=2$, 	then $G(M)$ is \CM.
	\end{theorem}
\begin{proof}
	Since $M$ has a free summand, we can write $M\cong N\oplus A^s$ for some $s\geq 1$ and $N$ has no free summand. We assume $N\ne 0$, otherwise  $M$ is free and  $G(M)$ is \CM.\\
	Clearly, $N$ is a MCM $A$-module (see \cite[Proposition 1.2.9]{BH}). Notice that red$(N)\leq 2$, because red$(M)\leq 2$. Also $\mu(M)> \mu(N)$, so $\mu(N)=1$.\\ 
	Now $N$ has  a minimal presentation  $0\rt Q\xrightarrow{a} Q\rt N\rt  0$, where $a\in \mathfrak{n}$. 
	This implies $N\cong Q/(a)Q$. So $G(N)$ is \CM.\\
		We know that (see \cite[Proposition 1.2.9]{BH})
	
	depth$G(M)\geq $ min\{depth$G(N)$, depth$G(A)$\}$=$ depth$G(N)$.\\
	This implies $G(M)$ is \CM.
		\end{proof}
	From Theorem \ref{muM=2} and Theorem \ref{mu2frsum} we can conclude:
	\begin{corollary}
		Let $(A,\mathfrak{m})$ be a complete hypersurface ring  of dimension $d$ with $e(A)=3$ and infinite residue field. Let  $M$ be  an MCM $A$-module with  $\mu(M)=2$, 
		then depth$G(M)\geq d-1$.
	\end{corollary}

\section{\bf The case when $\mu(M)=3$ }
In this section we prove Theorem \ref{2}. We first consider the case when $M$ has no free summand and prove:
\begin{theorem}\label{muM=3}
	Let $(A,\mathfrak{m})$ be a complete hypersurface ring  of dimension $d$ with $e(A)=3$ and infinite residue field. Let  $M$ be an MCM $A$-module with no free summand. Now if  $\mu(M)=3$, then depth$G(M)\geq d-2$.
	
\end{theorem}
\begin{proof} 
	From \ref{e(A)=3} we can take $(A,\mathfrak{m})=(Q/(f),\mathfrak{n}/(f))$, where
	$(Q,\mathfrak{n})$ is a regular local ring of dimension $d+1$ and $f\in \mathfrak{n}^3\setminus\mathfrak{n}^{4}$.\\
	Let dim$M\geq1$ and $0\rt Q^3\xrightarrow{\phi} Q^3\rt M\rt 0$ be a minimal presentation of $M$. Let $\underline{x}=x_1,\ldots,x_d$ be a maximal $\phi$-superficial sequence (see \ref{phi}). Set $M_d=M/\underline{x}M$ and $(Q',(y))=(Q/(\underline{x}),\mathfrak{n}/(\underline{x}))$.
	
	Clearly, $Q'$ is DVR and so $M_d\cong Q'/(y^{a_1})\oplus Q'/(y^{a_2})\oplus Q'/(y^{a_3})$. As $red_{(\underline{x})}(M)\leq2$ and $M_d$ has no free summand (see Lemma \ref{Md no free}), we can assume that $1\leq a_1\leq a_2\leq a_3\leq 2$. Now we  consider all possibilities separately.

	{\bf Case(1):} $a_1=1,a_2=1, a_3=1$.\\
	In this case $M_d\cong Q'/(y)\oplus Q'/(y)\oplus Q'/(y)$. This gives $h_{M_d}(z)=3$. So
	$e(M_d)=\mu(M_d)=3$.  For  dim$M\geq1$ we know that $e(M)=e(M_d)$ and $\mu(M)=\mu(M_d)$. Also, notice that $i(M)=i(M_d)=1$. So we have $e(M)=\mu(M)=3$ this implies that $M$ is an  Ulrich module.  So $G(M)$ is \CM \ and $h_M(z)=3$.  (see \cite[Theorem 2]{PuMCM}).
	
	{\bf Case(2):} $a_1=1,a_2=1, a_3=2$.\\
	In this case $M_d\cong Q'/(y)\oplus Q'/(y)\oplus Q'/(y^2)$. So we have $h_{M_d}(z)=3+z$.\\
	This implies  $v_{Q'}(det(\phi\otimes Q')) = v_{Q}(det(\phi))=4$. 
	Now for dim$M\geq 1$, from the Theorem \ref{mu(r)} depth$G(M)\geq d-1$. Also, we have two cases.\\ First case when  $h_M(z)=3+z$. In this case $G(M)$ is \CM.\\
	Second case when $h_M(z)=3+z^2$. In this case depth$G(M)=d-1$.

	{\bf Case(3):} $a_1=1,a_2=2, a_3=2$.\\
	In this case  $M_d\cong Q'/(y)\oplus Q'/(y^2) \oplus Q'/(y^2)$. So $h_{M_d}(z)=3+2z.$\\
	We  first consider the case when dim$M=3$ because if dim$M\leq2$ there is nothing to prove.\\
	Let $\underline{x}=x_1,x_2,x_3$ be a maximal $\phi$-superficial sequence. Set $M_1=M/x_1M$, $M_2=M/(x_1,x_2)M$, $M_3=M/(\underline{x})M$ and $J=(x_2,x_3)$ .\\
	We  first prove two claims:\\
	{\bf Claim(1):} $\widetilde{\mathfrak{m}^iM_2}=\mathfrak{m}^iM_2$ for all $i\geq2$.\\
	{\bf Proof} Since $\mathfrak{m}^{i+1}M_2=x_3\mathfrak{m}^iM_2$ for all $i\geq2$ so $\mathfrak{m}^{i+1}M_2:x_3=\mathfrak{m}^iM_2$ for all $i\geq2$.
	
	So from \ref{RR-2} for all $i\geq2$  we have
	$$0\rt \widetilde{\mathfrak{m}^iM_2}/\mathfrak{m}^iM_2\rt \widetilde{\mathfrak{m}^{i+1}M_2}/\mathfrak{m}^{i+1}M_2.$$ We also know that for $i\ggg 0$, $\widetilde{\mathfrak{m}^iM_2}=\mathfrak{m}^iM_2$. So it is clear that $\widetilde{\mathfrak{m}^iM_2}=\mathfrak{m}^iM_2$ for all $i\geq2$.\\
	{\bf Claim(2):} $\ell(\widetilde{\mathfrak{m}M_2}/\mathfrak{m}M_2)\leq1$. \\
	{\bf Proof:} Now since $\mu(M_2)=3$ so $\ell(\widetilde{\mathfrak{m}M_2}/\mathfrak{m}M_2)\leq3$. If $\ell(\widetilde{\mathfrak{m}M_2}/\mathfrak{m}M_2)=3$ then $M_2=\widetilde{\mathfrak{m}M_2}$ this implies that $\mathfrak{m}M_2=\mathfrak{m}\widetilde{\mathfrak{m}M_2}\sub\widetilde{\mathfrak{m}^2M_2}=\mathfrak{m}^2M_2$. So $\mathfrak{m}M_2=0$ and this is a contradiction. Now if possible assume that $\ell(\widetilde{\mathfrak{m}M_2}/\mathfrak{m}M_2)=2$  then $M_2=\langle m,l_1,l_2\rangle $ with $l_1,l_2\in \widetilde{\mathfrak{m}M_2}\setminus\mathfrak{m}M_2$ and $m\in M_2$.  Now we have  $\mathfrak{m}l_i\sub \widetilde{\mathfrak{m}^2M_2}=\mathfrak{m}^2M_2 $ for $i=1,2.$ If we set $\mathfrak{m}'=\mathfrak{m}/(x_1,x_2,x_3)$, then $\mathfrak{m}'$ is principal ideal. We also know that $\ell(\mathfrak{m}M_3)=\ell(\mathfrak{m}'M_3)$ and $\mathfrak{m}^2M_3=(\mathfrak{m}')^2M_3=0 $, this implies that $\ell(\mathfrak{m}M_3)=1$. But we have $\ell(\mathfrak{m}M_3)=2$ (from the Hilbert series of $M_3$). So $\ell(\widetilde{\mathfrak{m}M_2}/\mathfrak{m}M_2)\leq1$. 
	
	Now from \ref{RR-2} we have 
	$$0\rt (\mathfrak{m}^2M_2:x_3)/\mathfrak{m}M_2\rt \widetilde{\mathfrak{m}M_2}/\mathfrak{m}M_2 .$$
	So from claim(2), we have $b_1(x_3,M_2)=\ell(\mathfrak{m}^2M_2:x_3/\mathfrak{m}M_2)\leq 1$.\\
	Since dim$M_2=1$ we can write $h$-polynomial of $M_2$ as $h_{M_2}(z)=3+(\rho_0(M_2)-\rho_1(M_2))z+\rho_1(M_2)z^2$ where $\rho_n(M_2) =\ell(\mathfrak{m}^{n+1}M_2/{x_3\mathfrak{m}^nM_2})$ (see \ref{dm1hpol}). Since red$_{(\underline{x})}(M)\leq 2$, $\rho_n(M_2)=0$ for all $n\geq 2$. We have  $\rho_0(M_2)=\ell(\mathfrak{m}M_2/x_3M_2)= e(M_2)-\mu(M_2)=2$ and coefficients of $h_{M_2}$ are non-negative [from \ref{d=1}(1)].
	
	From short exact sequence (see \ref{exact d})
	$$0\rt \mathfrak{m}^2M_2:x_3/\mathfrak{m}M_2\rt \mathfrak{m}^2M_2/x_3\mathfrak{m}M_2\rt \mathfrak{m}^2M_3/0\rt 0$$
	we have $\rho_1=b_1(x_3,M_2)$ because $\mathfrak{m}^2M_3=0$. From Claim(2) we have $b_1(x_3,M_2)\leq1$.\\ 
	Now we have two cases.\\
	{\bf Subcase (i):} When  $b_1(x_3,M_2)=0$.\\
	This implies 
	$\rho_1(M_2)=0$ and so in this case $M_2$ has minimal multiplicity. Therefore $G(M_2)$ is \CM\ (see \ref{minmulM}).  By Sally-descent $G(M)$ is \CM \ and $h_M(z)=3+2z$.\\   
	{\bf Subcase (ii):} When $b_1(x_3,M_2)\neq0$.\\ From Claim(2) we have $\rho_1(M_2)=b_1(x_3,M_2)=1$. So in this case $h_{M_2}(z)=3+z+z^2$. Since $h_{M_2}(z)\ne h_{M_3}(z)$, depth$G(M_2)=0$ (see \ref{Property}). \\
		Since dim$M_1=2$, from \ref{Property}  we have  
		$$e_2(M_1)=e_2({M_2})-\sum b_i(x_2,M_1),$$
	where $\sum b_i(x_2,M_1)=\ell(\mathfrak{m}^{i+1}M_1:x_2/\mathfrak{m}^iM)$.\\
	We know that $e_2(M_1)$  and $\sum b_i(x_2,M_1) $ are non-negative integers (see \ref{e2>}). Also in this case $e_2({M_2})=1.$ So we have $\sum b_i(x_2,M_1)\leq 1$. 
	
	Since red$_{(\underline{x})}(M)\leq2$, from exact sequence (see \ref{exact seq}) 
	\begin{align*}
	0 \rt \mathfrak{m}^nM_1:J/\mathfrak{m}^{n-1}M_1\rt \mathfrak{m}^nM_1:x_2/\mathfrak{m}^{n-1}M_1 & \rt \mathfrak{m}^{n+1}M_1:x_2/\mathfrak{m}^{n}M_1\\
	\rt \mathfrak{m}^{n+1}M_1/J\mathfrak{m}^nM_1
	&\rt \mathfrak{m}^{n+1}{M_2}/x_3\mathfrak{m}^n{M_2}\rt 0
	\end{align*}
	we get, if $b_1(x_2,M_1)=0$ then  $b_i(x_2,M_1)=0$ for all $i\geq2$.
	
	{\bf Subcase (ii).(a):} When $\sum b_i(x_2,M_1) =0$.\\
	Now from \ref{Property},  depth$G(M_1)\geq1$. Notice that $G(M_1)$ cannot be a \CM\ module, because depth$G(M_2)=0$. So in this case depth$G(M_1)=1$. By Sally-descent depth$G(M)=2 $  
	and $h_M(z)=3+z+z^2$. 
	
	{\bf Subcase (ii).(b):} When $\sum b_i(x_2,M_1) \ne 0$.\\
	This implies $b_1(x_2,M_1)=\ell((\mathfrak{m}^2M_1:x_2)/\mathfrak{m}M_1)=1$.
	So in this case we have depth$G(M_1)=0$ (see \ref{Property}).

	From the above exact sequence we get 
	$$0\rt \mathfrak{m}^2M_1:x_2/\mathfrak{m}M_1\rt \mathfrak{m}^2M_1/J\mathfrak{m}M_1\rt \mathfrak{m}^2M_2/x_3\mathfrak{m}M_2\rt 0.$$
	
	So we have $\ell(\mathfrak{m}^2M_1/J\mathfrak{m}M_1)=\rho_1(M_2)+\ell((\mathfrak{m}^2M_1:x_2)/\mathfrak{m}M_1)=2$. We can write $h$-polynomial of $M_1$ as  $h_{M_1}(z)=h_{M_2}(z)-(1-z)^2z=3+3z^2-z^3$ (see \ref{Property}). 
	
	Consider $\overline{G(M)}=G(M)/(x_1^*,x_2^*,x_3^*)G(M)$. So we have
	$$\overline{G(M)}=G(M)/(x_1^*,x_2^*,x_3^*)G(M)=M/\mathfrak{m}M \oplus \mathfrak{m}M/(\underline{x})M \oplus \mathfrak{m}^2M/(\underline{x})\mathfrak{m}M$$
	Since $\ell(\mathfrak{m}M/(\underline{x})M)=2$ by looking at the Hilbert series of $\overline{G(M)}$ we can say that $\ell(\mathfrak{m}^2M/(\underline{x})\mathfrak{m}M)\leq2$ (see \ref{overline{G(M)}}).
	
	From exact sequence (see \ref{exact d}) $$0\rt (\mathfrak{m}^2M:x_1)/\mathfrak{m}M\rt \mathfrak{m}^2M/(\underline{x})\mathfrak{m}M\rt \mathfrak{m}^2M_1/J\mathfrak{m}M_1\rt 0$$
	we have $\ell(\mathfrak{m}^2M/(\underline{x})\mathfrak{m}M)=2$ and $\mathfrak{m}^2M:x_1=\mathfrak{m}M$ because $\ell(\mathfrak{m}^2M_1/J\mathfrak{m}M_1)=2$.
	
	Now consider $$\delta=\sum\ell(\mathfrak{m}^{n+1}M\cap (\underline{x})M/(\underline{x})\mathfrak{m}^nM).$$
	
	We know that if $\delta\leq2$ then depth$G(M)\geq d-\delta$ (see \cite[Theorem 5.1]{apprx}). In our case $\delta=2$, because $\mathfrak{m}^2M\sub (\underline{x})M$ and $\ell(\mathfrak{m}^2M/(\underline{x})\mathfrak{m}M)=2$. So depth$G(M)\geq 1$. Notice that in this case depth$G(M)=1$ because depth$G(M_1)=0$.\\ Now assume dim$M\geq4$ and $\underline{x}=x_1,\ldots,x_d$ a maximal $\phi$-superficial sequence. Set $M_{d-3}=M/(x_1,\ldots,x_{d-3})M$. We now have three cases.\\
	First case when $G(M_{d-3})$ is \CM\ and $h_{M_{d-3}}=3+2z$. By Sally-descent $G(M)$ is \CM\ and $h_M(z)=3+2z$.\\
	Second case when depth$G(M_{d-3})=2$. By Sally-descent depth$G(M)=d-1$ and $h_M(z)=3+z+z^2$.\\
	Third case when depth$G(M_{d-3})=1$. By Sally-descent depth$G(M)=d-2$ and $h_M(z)=3+3z^2-z^3$.

	{\bf Case(4):} $a_1=a_2=a_3=2$.\\
	In this case we have  $M_d\cong Q'/(y^2)\oplus Q'/(y^2)\oplus Q'/(y^2)$ and $h_{M_d}(z)=3+3z$. So we have $e(M_d)=i(M_d)\mu(M_d)=6$. For dim$M\geq1$ we know  $e(M)=i(M)\mu(M)=6$, because it is preserved modulo any $\phi$-superficial sequence.   This implies $G(M)$  is \CM\ and $h_M(z)=3+3z$  (see \cite[Theorem 2]{PuMCM}).
	\end{proof}

\begin{theorem}
	Let $(A,\mathfrak{m})$ be a complete hypersurface ring  of dimension $d$ with $e(A)=3$ and infinite residue field. Let  $M$ be an MCM $A$-module with   $\mu(M)=3$, then depth$G(M)\geq d-2$.
\end{theorem}
\begin{proof}
		We have two cases here.\\
	First case when $M$ has no free summand. In this case, from the above theorem depth$G(M)\geq d-2$.\\
	Next case when $M$ has  free summand. In this case we can write $M\cong N\oplus A^s$ for some $s\geq 1$ and $N$ has no free summand. We assume $N\ne 0$, otherwise  $M$ is free and  $G(M)$ is \CM.\\
	Clearly, $N$ is a MCM $A$-module (see \cite[Proposition 1.2.9]{BH}). Notice that red$(N)\leq 2$, because red$(M)\leq 2$. Also $\mu(M)> \mu(N)$, so $\mu(N)\leq 2$.\\ 
	If $\mu(N)=1$ then we have a minimal presentation of $N$ as $0\rt Q\xrightarrow{a} Q\rt N\rt  0$, where $a\in \mathfrak{n}$. 
	This implies $N\cong Q/(a)Q$. So, $G(N)$ is \CM.\\
	If $\mu(N)=2$ then depth$G(N)\geq d-1$ (from Theorem \ref{muM=2}).	
	 
	We know that (see \cite[Proposition 1.2.9]{BH})
	
	depth$G(M)\geq $ min\{depth$G(N)$, depth$G(A)$\}$=$ depth$G(N)$.\\
	So in this case depth$G(M)\geq d-1$. 
	\end{proof}

\section{Examples}\label{Examp} {\bf Case(1)} If $\mu(M)=2$ then we have

Take $Q=k[[x,y]]$, $\mathfrak{n}=(x,y)$
\begin{enumerate}
	
	\item $\phi = \begin{pmatrix}
	a & b \\
	c & d
	\end{pmatrix} $ with $ad-bc\ne0$ and $a,b,c,d\in \mathfrak{n}\setminus \mathfrak{n}^2$, then $G(M)$ is \CM \ and $h_M(z)=2$.
	\item $\phi = \begin{pmatrix}
	y^{a_1} & 0 \\
	0 & y^{a_2}
	\end{pmatrix} $ where $1\leq a_i\leq 3$ then $G(M)$ is \CM .
	\item $\phi = \begin{pmatrix}
	y^2 & 0 \\
	x^2 & y
	\end{pmatrix} $ then  $G(M)$ is \CM \ and $h_M(z)=2+z.$ Because if we set $e_1=(1,0)^T$ and $e_2=(0,1)^T$ then $M\cong (Q \oplus Q)/\langle y^2e_1+x^2e_2,ye_2 \rangle$. We can easily calculate $\ell(M/\mathfrak{n}M)=2$, $\ell(\mathfrak{n}M/\mathfrak{n}^2M)=3$ and $\ell(\mathfrak{n}^2M/\mathfrak{n}^3M)=3$. Now since $y^3M=0$ and dim$M=1$, we get deg$h_M(z)\leq 2$. From the above calculation it is clear $h_M(z)=2+z$ and this implies $M$ has minimal multiplicity. 
	\item $\phi = \begin{pmatrix}
	y^2 & 0 \\
	x & y
	\end{pmatrix} $ then  depth$G(M)=0$    and $h_M(z)=2+z^2.$ Because if we set $e_1=(1,0)$ and $e_2=(0,1)$ then $M\cong (Q\oplus Q)/\langle y^2e_1+xe_2,ye_2\rangle$. Now it is clear that $\overline{e_2}\in \widetilde{\mathfrak{n}M}\setminus \mathfrak{n}M$. So $\widetilde{\mathfrak{n}M} \ne \mathfrak{n}M$ and this implies $G(M)$ has depth zero. Also, dim$M=1$ and $\rho_2(M)=0$ so, deg$h_M(z)\leq2$. It is also clear that $\ell(M/\mathfrak{n}M)=2$, $\ell(\mathfrak{n}M/\mathfrak{n}^2M)=2$ because $x\overline{e_2}=y^2\overline{e_1}$. Similar calculation gives that $\ell(\mathfrak{n}^2M/\mathfrak{n}^3M)=3$ and $y^3M=0$. So, $h_M(z)=2+z^2$.
	
\end{enumerate}

{\bf Case(2)} Now if $\mu(M)=3$ then examples are:\\
Take $Q=k[[x,y,z]]$ and $\mathfrak{n}=(x,y,z)$.
\begin{enumerate}
	\item   $\phi= \begin{pmatrix}
	x & y & z\\
	x^2 & x^2 & 0\\
	0 & 0 & x^2
	\end{pmatrix} $ then $x\overline{e_1}= -x^2\overline{e_2}$, $y\overline{e_1}=-x^2\overline{e_2}$ and $z\overline{e_1}=-x^2\overline{e_3}$. From here we get $(x-y)\overline{e_1}=0$, $x^2(x-y)\overline{e_2}=(x-y)(-x)\overline{e_1}=0$ and $x^2(x-y)\overline{e_3}=(x-y)(-z)\overline{e_1}=0$. Therefore, $x^2(x-y)\overline{e_i}=0$ for $i=1,2,3$ and $\overline{e_1}\in \widetilde{\mathfrak{n}M}\setminus\mathfrak{n}M.$. So  $M$ is $Q/(x^2(x-y))$-module and depth$G(M)=0$.
	\item  $\phi= \begin{pmatrix}
	x & y & 0\\
	x^2 & x^2 & 0\\
	0 & 0 & x^2
	\end{pmatrix} $ then depth$G(M)=1$, because $z^*$ is $G(M)-$regular and after going modulo $z^*$ we get depth${G({{N}})}=0$, here ${N}=M/zM$. Notice that $\overline{e_1}\in \widetilde{\mathfrak{n}N}\setminus\mathfrak{n}N$.  Here $M$ is $Q/(x^2(x-y))$-module, because $x^2(x-y)\overline{e_i}=0$ for $i=1,2,3.$
	\item  $\phi= \begin{pmatrix}
	x & 0 & 0\\
	0 & x^2 & 0\\
	0 & 0 & x^2
	\end{pmatrix} $ then depth$G(M)=2$, i.e. $G(M)$ is \CM.
\end{enumerate}

{\bf Case(3):} If $\mu(M)=r$; take  $Q=k[[x,y]]$, $\mathfrak{n}=(x,y)$ 	
(This is also the case when $e(M)=\mu(M)i(M)+1$).
\begin{enumerate}
	\item $[\phi]_{r\times r}= \begin{pmatrix}
	y^2 & 0 &0 &\cdots & 0\\
	x^2 & y & 0 &\cdots & 0\\
	0   & 0 & y &\cdots & 0\\
	\vdots & \vdots &\vdots& \ddots &0\\
	0      &  0     &  0   & \cdots & y
	\end{pmatrix} $ then $det\in\mathfrak{n}^{r+1}\setminus\mathfrak{n}^{r+2}$, $G(M)$ is \CM \ and $h_M(z)=r+z.$
	
	\item  $[\phi]_{r\times r}= \begin{pmatrix}
	y^2 & 0 &0 &\cdots & 0\\
	x & y & 0 &\cdots & 0\\
	0   & 0 & y &\cdots & 0\\
	\vdots & \vdots &\vdots& \ddots &0\\
	0      &  0     &  0   & \cdots & y
	\end{pmatrix} $then $det\in\mathfrak{n}^{r+1}\setminus\mathfrak{n}^{r+2}$, depth$G(M)=0$ because $\widetilde{\mathfrak{n}M}\neq \mathfrak{n}M$ as $\overline{e_2}\in \widetilde{\mathfrak{n}M}\setminus\mathfrak{n}M$.  and $h_M(z)=r+z^2.$

\end{enumerate}


\begin{thebibliography}{10}
	\bibitem{BH}
	W.~Bruns and J.~Herzog, \emph{{Cohen-Macaulay rings}}, vol.~39, Cambridge
	studies in advanced mathematics, Cambridge University Press,~Cambridge, 1993.
	
     \bibitem{Eisenbud}
    D. Eisenbud, \emph{Homological algebra on complete intersections with an application to group representations}, Trans. Amer. Math. Soc., 260 (1980), 35-64.
     
     	\bibitem{heinzer}
     W.~Heinzer, B.~Johnston, D.~Lantz, K.~Shah,
     \emph{The Ratliff -Rush ideals in a Noetherian ring: a survey, in: Methods in Module Theory}(Colorado
     Springs, CO, 1991), in: Lecture Notes in Pure and Appl. Math., vol. 140, Dekker, New York, 1993, pp. 149–159.
     
     
     \bibitem{Naghipour}
     R.~Naghipour,
     \emph{Ratliff-Rush closures of ideals with respect to a Noetherian module},J. Pure Appl. Algebra 195 (2) (2005) 167–172.
     
	\bibitem{Pu0}
		T.~J. Puthenpurakal,
	\emph{Hilbert coefficients of a Cohen--Macaulay module}, J. Algebra \textbf{264}  (2003), no.~1, 82--97.
	
	
	\bibitem{PuMCM}
	\bysame
	\emph{The Hilbert function of a maximal Cohen-Macaulay module}, Math. Z. \textbf{251} (2005), no.~3, 551--573.
	
	
	
	\bibitem{Pu1}
		\bysame,
	\emph{Ratliff-{R}ush filtration, regularity and depth of higher associated graded modules. {I}},
	J. Pure Appl. Algebra \textbf{208} (2007), no.~1, 159--176.
	
	\bibitem{apprx}
	  \bysame,
	  \emph{Complete intersection approximation, dual filtrations and applications}  arXiv preprint arXiv:0807.0471,(2008).
	
	\bibitem{Pu2}
		\bysame,
	\emph{Ratliff-{R}ush filtration, regularity and depth of higher associated graded modules. {II}},
	J. Pure Appl. Algebra \textbf{221} (2017), no.~3,
	611--631.
	
		\bibitem{Ratliff}
	L.~J.~Ratliff, D.~Rush,
	\emph{Two notes on reductions of ideals}, Indiana Univ. Math. J. 27 (1978) 929–934.
	
		\bibitem{rv}
	M.~E.~Rossi and G.~Valla,
	\emph{A conjecture of J. Sally},
	Comm. Algebra 24 (1996), no. 13, 4249–-4261.
	
	\bibitem{Rossi}
	 Maria Evelina Rossi and  Giuseppe Valla,
	\emph{Hilbert functions of filtered modules}, vol.9, Springer Science \& Business Media,2010
	
	\bibitem{Sbook}
	J.~D.~Sally
	\emph{Number of generators of ideals in local rings}, Lect. Notes Pure Appl. Math., vol.~35, M. Dekker, 1978.
		\bibitem{S}
	\bysame,
	\emph{Tangent cones at Gorenstein singularities},
	Compositio Math. 40 (1980) 167–-175.
	
	\bibitem{singh}
	B.~Singh,
	\emph{Effect of a permissible blowing-up on the local Hilbert functions}, Invent. Math. 26(1974), 201-212.
\end{thebibliography}
\end{document}